\documentclass{amsart}

\title[The Diederich--Fornaess exponent and conformal harmonic measures]
{A local expression of the Diederich--Fornaess exponent and the exponent of conformal harmonic measures}
\author{Masanori Adachi}
\address{Graduate~School~of~Mathematics, Nagoya~University, Furo-cho Chikusa-ku Nagoya 464-8602, Japan}
\email{m08002z@math.nagoya-u.ac.jp}
\curraddr{Center~for~Geometry~and~its~Applications, Pohang~University~of~Science~and~Technology, San 31 Hyoja-dong Nam-gu Pohang 790-784, Republic of Korea
}

\date{\today}
\thanks{This is the author's final version of an article published as Bull. Braz. Math. Soc. (N.S.) {\bf 46} (2015), no. 1, 153--167.
The final publication is available at www.springerlink.com}

\usepackage{amsmath,amsthm,amssymb,latexsym}
\usepackage[initials]{amsrefs}

\newtheorem*{Thm*}{Theorem}

\newtheorem{Thm}{Theorem}[section]
\newtheorem{Prop}[Thm]{Proposition}

\theoremstyle{definition}
\newtheorem{Def}[Thm]{Definition}
\newtheorem{Ex}[Thm]{Example}

\newtheorem{Cor}[Thm]{Corollary}
\newtheorem{Lem}[Thm]{Lemma}

\newtheorem*{claim}{Claim}

\theoremstyle{remark}
\newtheorem{Rem}[Thm]{Remark}

\def\Z{\mathbb{Z}}
\def\R{\mathbb{R}}
\def\C{\mathbb{C}}
\def\D{\mathbb{D}}
\def\Cont{\mathcal{C}}
\def\CP{\mathbb{CP}}

\newcommand{\eqdef}{\mathrel{\mathop:}=}

\def\Re{\mathrm{Re}\,}
\def\Im{\mathrm{Im}\,}
\def\zbar{\overline{z}}

\def\d'{\partial}
\def\dbar{\opa}
\def\dbar{\overline{\partial}}
\def\bd{\partial}
\newcommand{\ol}{\overline }

\allowdisplaybreaks 

\makeatletter
\newsavebox{\@brx}
\newcommand{\llangle}[1][]{\savebox{\@brx}{\(\m@th{#1\langle}\)}%
  \mathopen{\copy\@brx\kern-0.5\wd\@brx\usebox{\@brx}}}
\newcommand{\rrangle}[1][]{\savebox{\@brx}{\(\m@th{#1\rangle}\)}%
  \mathclose{\copy\@brx\kern-0.5\wd\@brx\usebox{\@brx}}}
\makeatother

\begin{document}

\maketitle

\begin{abstract}
A local expression of the Diederich--Fornaess exponent 
of complements of Levi-flat real hypersurfaces is exhibited. 
This expression describes the correspondence between pseudoconvexity of 
their complements and positivity of their normal bundles, 
which was suggested in a work of Brunella, in a quantitative way.  
As an application, a connection between the Diederich--Fornaess exponent 
and the exponent of conformal harmonic measures is discussed.
\end{abstract}

\section{Introduction}
A smooth real hypersurface $M$ in a complex manifold $X$ is said to be Levi-flat if 
it separates the ambient manifold $X$ locally into two Stein domains. 
From Frobenius' theorem, it is equivalent for $M$ to have a foliation $\mathcal{F}$, called the Levi foliation, 
by non-singular complex hypersurfaces of $X$.
By its definition, therefore, we have two viewpoints to study Levi-flat real hypersurfaces, 
namely, extrinsically from several complex variables and intrinsically from foliation theory. 

One of the indications of such situation was observed by Brunella \cite{brunella2008}. 
He proved that, under some regularity conditions, Takeuchi 1-convexity of the complement 
of a smooth compact Levi-flat real hypersurface follows from the positivity of its normal bundle $N^{1,0}$.
We can also go through its converse actually, and 
are able to establish the correspondence between these qualitative properties.

These two notions have their quantitative counterparts:
Takeuchi 1-convexity is characterized by the existence of a boundary distance function whose
Diederich--Fornaess exponent in a weak sense is positive, 
according to a work of Ohsawa and Sibony \cite{ohsawa-sibony1998};
The positivity of the normal bundle is 
obviously determined by the existence of a hermitian metric with positive Chern curvature along the leaves. 
We can therefore expect a quantitative correspondence between the Diederich--Fornaess exponent and 
the Chern curvature along the leaves, and this is the main question of this paper. 

Our answer is the following local expression that relates not only the two quantities 
but also an additional curvature-like term denoted by $iA_h$, which has appeared in literature on foliation theory
(cf. Frankel \cite{frankel1995}, Deroin \cite{deroin2005}).

\begin{Thm}
\label{localformula}
Let $X$ be a complex manifold and $\Omega \Subset X$ a relatively compact domain 
with $\Cont^3$-smooth Levi-flat boundary $M$.
Suppose that $\Omega$ has a boundary distance function $\delta$ with positive Diederich--Fornaess exponent in the weak sense
and denote by $h_\delta$ the hermitian metric of the normal bundle $N^{1,0}$ induced from $\delta$. 
Then, the local Diederich--Fornaess exponent $\eta_\delta(p)$  
is expressed as
\[
\eta_\delta(p) =  \sup \{ \eta \in (0,1) \mid i\Theta_{h_\delta}(p) - \frac{\eta}{1-\eta} iA_{h_\delta}(p) > 0 \}
\]
where $i\Theta_{h_\delta} \eqdef -i\d'_b\dbar_b \log h_\delta$ is the Chern curvature of $h_\delta$ along the leaves
and $iA_{h_\delta} \eqdef i\d'_b\log {h_\delta} \wedge \dbar_b \log {h_\delta}$. 
\end{Thm}

In particular, we have

\begin{Cor}
Suppose $\dim_\C X =2$ and $\Omega, M, \delta$ and $h_\delta$ as in Theorem \ref{localformula}. 
Then, the local Diederich--Fornaess exponent $\eta_\delta(p)$ is determined by 
the ratio of $i\Theta_{h_\delta}(p)$ and $iA_{h_\delta}(p)$, namely, 
\[
\eta_\delta(p) =  \left(1 + \frac{iA_{h_\delta}(p)}{i\Theta_{h_\delta}(p)}\right)^{-1}.
\]
\end{Cor}

\bigskip

As an application of this local formula, we give an observation  
that indicates a connection between 
boundary distance functions with constant local Diederich--Fornaess exponent and 
conformal harmonic measures with smooth density.

\begin{Prop}
\label{relation}
Let $X$ be a complex surface and $\Omega \Subset X$ a relatively compact domain 
with $\Cont^{3}$-smooth Levi-flat boundary $M$.
We fix a continuous leafwise hermitian metric on the Levi foliation. 
Suppose that $\Omega$ possesses a boundary distance function $\delta$ 
whose local Diederich--Fornaess exponent $\eta_\delta(p)$ is constant equal to $\eta$ on 
a compact saturated set $\mathcal{M} \subset M$ of the Levi foliation.
Then, $\mathcal{M}$ admits a conformal harmonic measure with exponent $\alpha = (\eta^{-1} - 1)^{-1}$
if transversals of $\mathcal{M}$ has positive finite Hausdorff measure of dimension exactly $\alpha$.
\end{Prop}

\bigskip

The organization of this paper is as follows. 
In \S\ref{Sprelim}, we give the definition of the local Diederich--Fornaess exponent 
and explain how to construct the hermitian metric $h_\delta$ which appeared in the statement of Theorem \ref{localformula}.
In \S\ref{Slocalformula}, we show that the converse of Brunella's theorem actually holds in a sense
and prove Theorem \ref{localformula}. 
Their proofs are by direct computations in a suitable coordinate, 
which is referred to as a distinguished parametrization in this paper 
and based on Diederich and Fornaess \cite{diederich-fornaess-worm} and Barrett and Fornaess \cite{barrett-fornaess1988}. 
In \S\ref{Srelation}, after recalling the notion of conformal harmonic measure, 
we prove Proposition \ref{relation} and illustrate it by an example.

\section{Preliminaries}
\label{Sprelim}

\subsection{The Diederich--Fornaess exponent}
Let us begin by recalling the definition of the Diederich--Fornaess exponent. 

Let $X$ be a complex manifold and $\Omega \Subset X$ a relatively compact domain with $\Cont^2$-smooth boundary.
We say that a $\Cont^2$-smooth function $\delta: \overline{\Omega} \to \R$ is 
a {\em boundary distance function} of $\Omega$ if $\Omega = \{\delta > 0 \}$ and $d\delta \neq 0$ on $\bd \Omega$. 

\begin{Def}
We denote by $\eta_\delta$ the supremum of $\eta \in (0, 1)$ such that 
$-\delta^{\eta}$ is strictly plurisubharmonic in $\Omega$ except a compact subset, 
and call it the {\em Diederich--Fornaess exponent} (in the weak sense) of $\delta$. 
If there is no such $\eta$, we let $\eta_\delta = 0$.
\end{Def}

Note that we usually require $-\delta^{\eta}$ to be strictly plurisubharmonic everywhere in $\Omega$ 
to define the Diederich--Fornaess exponent, but in this paper we will consider 
the Diederich--Fornaess exponent in this weak sense; we will omit mentioning that it is in the weak sense hereafter.

\bigskip

We denote by $\eta(\Omega)$ the supremum of the Diederich--Fornaess exponents 
of boundary distance functions of $\Omega$ and call it the {\em Diederich--Fornaess index} of $\Omega$.

\begin{Ex}
$\eta(\Omega) = 1$ for any domain $\Omega \Subset X$ with strongly pseudoconvex boundary. 
Because we can find a strictly plurisubharmonic defining function of $\bd\Omega$.
\end{Ex}

\begin{Ex}[Diederich and Fornaess \cite{diederich-fornaess1977}] 
$\eta(\Omega) > 0 $ for any $\Cont^2$-smoothly bounded pseudoconvex domains $\Omega \Subset X$ whenever $X$ is Stein.
\end{Ex}

Ohsawa and Sibony gave a characterization for $\Omega$ to have positive $\eta(\Omega)$. 

\begin{Thm}[Ohsawa and Sibony \cite{ohsawa-sibony1998}; see also Harrington and Shaw \cite{harrington-shaw2007}]
Let $X$ be a complex manifold and $\Omega \Subset X$ a relatively compact domain with $\Cont^2$-smooth boundary.
Then, a boundary distance function $\delta$ satisfies $\eta_\delta > 0$ 
if and only if $i\d'\dbar (-\log \delta) \geq \omega$ in $\Omega$ except a compact subset for some hermitian metric $\omega$ of $X$; 
in other words, $\eta(\Omega) > 0$ if and only if $\Omega$ is Takeuchi 1-convex. 
\end{Thm}

\begin{Rem}
Ohsawa proved in \cite{ohsawa2007} that $\eta(\Omega) = 0$ for $\Omega$ with $\Cont^\infty$-smooth Levi-flat boundary
whenever $X$ is a compact K\"ahler manifold of $\dim_\C X \geq 3$.
Nevertheless, in some compact K\"ahler surface $X$, $\eta(\Omega)$ can be positive even if $\Omega$ has Levi-flat boundary. 
See \cite{diederich-ohsawa2007} and \cite{adachi2014}.
\end{Rem}

To compare the Diederich--Fornaess exponent of $\Omega$ with certain curvatures of the normal bundle of the Levi-flat boundary $\bd\Omega$, 
we need the following pointwise notion regarding the Diederich--Fornaess exponent.

\begin{Def}
Let $\delta$ be a boundary distance function of $\Omega$. 
For each $p \in \bd\Omega$, we denote by $\eta_\delta(p)$ the supremum of $\eta \in (0, 1)$ such that 
$-\delta^\eta$ is strictly plurisubharmonic on $U \cap \Omega$ where $U$ runs over neighborhoods of $p$ in $X$, 
and call it \emph{the local Diederich--Fornaess exponent} of $\delta$ at $p \in \bd \Omega$. 
If there is no such $\eta$, we let $\eta_\delta(p) = 0$.
\end{Def}

Note that the local Diederich--Fornaess exponent $\eta_\delta(p)$ is a lower semi-continuous function on $\bd \Omega$
from its definition and $\eta_\delta = \min_{p \in \bd\Omega} \eta_\delta(p)$.

\begin{Rem}
Herbig and McNeal \cite{herbig-mcneal2012} gave a lower estimate of the local Diederich--Fornaess exponent
of the Euclidean boundary distance function of $\Omega \Subset \C^n$ using 
the angle of a certain cone in the real tangent space $T_p\bd\Omega$. 
We remark that their lower estimate turns into a dichotomy at weakly pseudoconvex points, 
where $\gamma$-plurisubharmonicity holds either only for $\gamma = 0$ or for any $\gamma \geq 0$
(cf. \cite[Remark 6.2 (b)]{herbig-mcneal2012}).
\end{Rem}

\begin{Rem}
Biard \cite[Th\'eor\`eme 2.12]{biardthese} gave a lower estimate of the local Diederich--Fornaess exponent
using a function called $\tau$. Our Theorem \ref{localformula} can be regarded as a refinement of her result 
in terms of boundary invariants when $\bd \Omega$ is $\Cont^3$-smooth Levi-flat. 
\end{Rem}

\begin{Rem}
Fu and Shaw \cite{fu-shaw2014} gave other lower estimates of 
the local Diederich--Fornaess exponent in terms of several geometric quantities.  
\end{Rem}

\subsection{The normal bundle}

Let us explain how to make a hermitian metric of the normal bundle
of a Levi-flat real hypersurface from a boundary distance function of its complement. 

From now on, we work with $\Omega \Subset X$ with $\Cont^3$-smooth Levi-flat boundary $M$. 
We denote the {\em holomorphic tangent bundle} of $M$ by $T\mathcal{F} \eqdef TM \cap J_X TM$ 
where $J_X$ denotes the complex structure of $X$. We identify the subbundle $T\mathcal{F} \subset TX|M$ with 
$T^{1,0}M \subset T^{1,0}X|M$ via the standard identification. The bundle $T^{1,0}M$ is called the {\em CR structure} of $M$.
The Levi-flatness of $M$ implies the integrability of $T\mathcal{F}$ in the sense of Frobenius, 
and defines a foliation $\mathcal{F}$ of $M$ by non-singular complex hypersurfaces 
called the {\em Levi foliation}. 
The foliation $\mathcal{F}$ is not only $\Cont^2$-smooth as expected from Frobenius' theorem 
but also $\Cont^3$-smooth according to a theorem of Barrett and Fornaess \cite{barrett-fornaess1988}.

\begin{Def}
We call $N^{1,0} \eqdef \C \otimes (TM/T\mathcal{F})$ the {\em normal bundle} of a Levi-flat real hypersurface $M$.
\end{Def}

The normal bundle $N^{1,0}$ is an intrinsic notion on Levi-flat real hypersurfaces, 
and the trivialization cover of $N^{1,0}$ coming from foliated charts of $M$ shows that $N^{1,0}$ is a leafwise flat $\C$-bundle.
Note that $N^{1,0}$ is different from the normal bundle $N_{M/X}$ of a real hypersurface $M$ in $X$; 
the connection of $N^{1,0}$ with the ambient space $X$ can be described by an isomorphism $N^{1,0} \simeq (T^{1,0}X|M) / T^{1,0}M$.
This expression clearly shows that $N^{1,0}$ is a CR line bundle. \\

Now let $\delta$ be a $\Cont^3$-smooth boundary distance function of $\Omega$. 
We take an arbitrary foliated chart with parametrization $\varphi: \C^{n-1}\times \R \supset V \to U \subset M$.
Note that $\varphi(z',t)$ is $\Cont^3$-smooth and holomorphic in $z'$.
We define a $\Cont^2$-smooth positive function $h_\delta$ on $V$ by
\[
h_\delta(z', t) \eqdef \left| \left(J_X \varphi_*\frac{\d'}{\d't}\right) \delta \right|.
\]

\begin{Prop}
\label{construction}
The function $h_\delta$ defines a hermitian metric\footnote{In this paper, we use a convention where $h_\delta$ measures the norm of a normal vector, not its squared norm.} of $N^{1,0}$. 
\end{Prop}

\begin{proof}
Take two intersecting foliated charts and consider their parametrizations for the intersection $U$,  
say $\varphi_k: \C^{n-1}\times \R \supset V_k \to U \subset M$ $(k = \alpha, \beta)$.
We denote their transition function by $\psi = (\varphi_\beta)^{-1} \circ \varphi_\alpha: V_\alpha \to V_\beta$ and 
denote $\psi(z'_\alpha, t_\alpha) = (z'_\beta(z'_\alpha, t_\alpha), t_\beta(t_\alpha))$. By the chain rule, 
\[
\psi_*\frac{\d'}{\d't_\alpha} = \sum_{j=1}^{n-1} \left(\frac{\d'x_j}{\d't_\alpha} \frac{\d'}{\d'x_j} + \frac{\d'y_j}{\d't_\alpha} \frac{\d'}{\d'y_j}\right)+ \frac{dt_\beta}{dt_\alpha} \frac{\d'}{\d't_\beta}
\]
where we simply denote $(z'_\beta)_j = x_j + i y_j$. We have  
\begin{align*}
\left(J_X (\varphi_\alpha)_*\frac{\d'}{\d't_\alpha}\right) \delta
& = \left(J_X (\varphi_\beta\circ \psi)_* \frac{\d'}{\d't_\alpha}\right) \delta  \\
& = \sum_{j=1}^{n-1} (\varphi_\beta)_* \left(\frac{\d'x_j}{\d't_\alpha} \frac{\d'}{\d'y_j} - \frac{\d'y_j}{\d't_\alpha} \frac{\d'}{\d'x_j}\right)\delta + \frac{dt_\beta}{dt_\alpha} \left(J_X (\varphi_\beta)_* \frac{\d'}{\d't_\beta}\right) \delta \\
& = \frac{dt_\beta}{dt_\alpha} \left(J_X (\varphi_\beta)_* \frac{\d'}{\d't_\beta}\right) \delta 
\end{align*}
since $(\varphi_\beta)_* \frac{\d'}{\d'x_j}, (\varphi_\beta)_* \frac{\d'}{\d'y_j}$ are tangent to $M = \{ \delta = 0 \}$. 
This equality exactly shows that $h_\delta$ defines a hermitian metric of $N^{1,0}$.
\end{proof}

\section{Brunella's correspondence and its quantitative version} 
\label{Slocalformula}

\subsection{Distinguished parametrizations}
We will introduce first a special coordinate, which is essential to prove Proposition \ref{converse} and Theorem \ref{localformula}. 

We continue to work with $\Omega \Subset X$ with $\Cont^3$-smooth Levi-flat boundary $M$ 
and denote the Levi foliation by $\mathcal{F}$.
Take a point $p \in M$ and consider a holomorphic coordinate $(U; z=(z_1, z_2, \cdots, z_n) )$ around $p$.
We write $z_j = x_j + i y_j$ and identify $p$ with the origin.  
We can easily see that the holomorphic coordinate can be chosen so that
$L_p \cap U = \{ z_n = 0 \}$,
$T_p M \simeq \{ y_n = 0\}$ and $\frac{\d'}{\d' y_n}$ is an inner normal vector for $\Omega$ at $p$ 
where $L_p$ denotes the leaf of $\mathcal{F}$ passing through $p$.
Furthermore, we can make the holomorphic coordinate to satisfy 
$T_q M \simeq \{ y_n = 0 \}$ for any $q \in L_p \cap U$
thanks to the argument used in Diederich and Fornaess \cite{diederich-fornaess-worm}*{Theorem 8} 
and Barrett and Fornaess \cite{barrett-fornaess1988}*{Proposition}. 

Since $\mathcal{F}$ is $\Cont^3$-smooth, 
we have a foliated chart of $\mathcal{F}$ around $p$ 
whose parametrization $\varphi$ into the holomorphic coordinate $U$ is described as $\varphi(z', t) =  (z', \zeta(z', t))$
where $\zeta: \C^{n-1}\times \R \supset V \to \C$ is a $\Cont^3$-smooth function defined near $(0',0)$
that is holomorphic in $z'$ and $\Re \zeta(0',t) = t$. 
Note that $\zeta(z',0) = 0$ and $\frac{\d'\zeta}{\d't}(z', 0) = 1$ from the choice of the holomorphic coordinate.
The latter follows from $\Im \frac{\d'\zeta}{\d't}(z', 0) = 0 $, the holomorphicity of $\frac{\d'\zeta}{\d't}$ in $z'$ and $\frac{\d'\zeta}{\d't}(0', 0) = \frac{\d'\Re \zeta}{\d't}(0', 0) = 1$. 
We refer to this parametrization $\varphi$ in the coordinate $(U,z)$ as a {\em distinguished parametrization} around $p$.

\subsection{Curvatures of the normal bundle}
To state Proposition \ref{converse} and Theorem \ref{localformula} precisely, we introduce two notions of curvature of the normal bundle.

We denote by $\d'_b$ (resp. $\dbar_b$) the (1,0)-part (resp. (0,1)-part) of the exterior derivative along the leaves, 
namely, 
\[
\d'_b f = \sum_{j=1}^{n-1} \frac{\d' f}{\d' z_j} dz_j, \quad \dbar_b f = \sum_{j=1}^{n-1} \frac{\d' f}{\d' \zbar_j} d\zbar_j
\]
for a function $f$ on a foliated chart with coordinate $(z'=(z_1, z_2, \cdots, z_{n-1}), t)$. 
The operator $\dbar_b$ is also known as the tangential Cauchy--Riemann operator.

\begin{Def} For a (leafwise) $\Cont^2$-smooth hermitian metric $h$ of $N^{1,0}$, 
we let 
$i\Theta_{h} \eqdef -i\d'_b\dbar_b \log h$ and call it {\em the Chern curvature of $h$ along the leaves}. 
We also denote $iA_{h} \eqdef i\d'_b\log h \wedge \dbar_b \log h$.

We say $N^{1,0}$ is {\em positive along the leaves}  
if $i\Theta_{h}$ defines a positive definite quadratic form on $T^{1,0}M$ everywhere for some $h$. 
\end{Def}

Note that the leafwise flatness of $N^{1,0}$ assures the well-definedness of $iA_h$;
on the other hand, the Chern curvature along the leaves can be defined in fact 
for arbitrary CR line bundle over a Levi-flat real hypersurface equipped with a (leafwise) $\Cont^2$-smooth hermitian metric.

\subsection{Brunella's correspondence}
Now we rephrase a theorem of Brunella in our context by using Ohsawa--Sibony's theorem. 
\begin{Thm}[Brunella \cite{brunella2008}]
Let $X$ be a complex manifold and $\Omega \Subset X$ a relatively compact domain 
with $\Cont^{2,\alpha}$-smooth Levi-flat boundary $M$ where $\alpha > 0$. 
Suppose that $M$ is a saturated set of a non-singular holomorphic foliation defined near $M$.
If the normal bundle $N^{1,0}$ is positive along the leaves, 
there is a boundary distance function $\delta$ of $\Omega$ with positive Diederich--Fornaess exponent.
\end{Thm}


Let us prove the converse of Brunella's theorem in the following sense.
\begin{Prop}
\label{converse}
Let $X$ be a complex manifold and $\Omega \Subset X$ a relatively compact domain 
with $\Cont^{3}$-smooth Levi-flat boundary $M$.
If there is a boundary distance function $\delta$ of $\Omega$ with positive Diederich--Fornaess exponent,
the normal bundle $N^{1,0}$ is positive along the leaves.
\end{Prop}

\begin{proof}
We equip $N^{1,0}$ with the hermitian metric $h_\delta$ induced from $\delta$. 
Take a point $p \in M$ and we will show that $i\Theta_{h_\delta}(p)$ defines a positive definite quadratic form on $T^{1,0}_p M $. 

Take a distinguished parametrization $\varphi = (z', \zeta(z',t))$ in $(U, z = (z_1, z_2, \cdots, z_n))$ around $p$. 
Since $\frac{\d'\zeta}{\d't}(z', 0) = 1$,
\[
h_\delta(z', 0) 
= \left(J_X \varphi_*\frac{\d'}{\d't}\right) \delta  
= \left(J_X \frac{\d'}{\d'x_n}\right) \delta = \frac{\d'\delta}{\d' y_n} (z', 0).
\]
Hence, 
\[
i\Theta_{h_\delta}(p) 
 = \sum_{j,k=1}^{n-1} \frac{\d'^2}{\d'z_j \d'\zbar_k} \left(-\log \frac{\d'\delta}{\d' y_n} \right) (0', 0)\, i dz_j \wedge d\zbar_k.
\]

On the other hand, from Ohsawa--Sibony's theorem, we have a hermitian metric $\omega$ on $X$ so that 
$i\d'\dbar (-\log \delta) > \omega$ in $\Omega$ except a compact subset; in particular, 
\[
\sum_{j,k=1}^{n-1} \frac{\d'^2 (-\log \delta)}{\d'z_j \d'\zbar_k}(0', iy_n)\, i dz_j \wedge d\zbar_k > \omega|T^{1,0}_{(0', iy_n)}(\C^{n-1}\times \{iy_n\})
\]
holds for $0 < y_n \ll 1$. Since 
\[
\sum_{j,k=1}^{n-1} \frac{\d'^2(-\log \delta)}{\d'z_j \d'\zbar_k}(0', iy_n)\, i dz_j \wedge d\zbar_k
= \sum_{j,k=1}^{n-1} \frac{\d'^2(-\log (\delta/y_n))}{\d'z_j \d'\zbar_k}(0', iy_n)\, i dz_j \wedge d\zbar_k,
\]
holds for $0 < y_n \ll 1$ and $\delta(z', iy_n)/y_n$ extends to $\frac{\d' \delta}{\d' y_n}(z', 0)$ as a $\Cont^2$-smooth function,
letting $y_n \searrow 0$, we conclude with 
\[
i\Theta_{h_\delta}(p) \geq \omega|T^{1,0}_p M > 0.
\]
\end{proof}

\begin{Rem}
This observation gives a negative answer to Question 1 of \cite{adachi2014}. 
If the complement of a Levi-flat real hypersurface is Takeuchi 1-convex, 
its Levi foliation cannot contain any compact leaf $L$. 
That is because Proposition \ref{converse} yields the positivity of $N^{1,0}$, 
therefore, $N^{1,0}|L$ is also positive and $[c_1(N^{1,0}|L)] \neq 0 $ in $H^2(L;\Z)$; 
but, $N^{1,0}|L$ is topologically trivial and $[c_1(N^{1,0}|L)] = 0$. This is absurd.
Therefore, in the cases (iii) and (iv) in \cite{adachi2014}*{Theorem 3.1}, 
their associated holomorphic disc bundles cannot be Takeuchi 1-convex. 
\end{Rem}

\subsection{Expression of the local Diederich--Fornaess exponent}

We are going to prove Theorem \ref{localformula}. 
Although the idea of its proof is similar to that of Proposition \ref{converse}, 
this time we will deal with the complex hessian matrix of $-\delta^\eta$ instead of that of $-\log \delta$
and this change requires us to look at not only $z'$-directions but also $z_n$-direction.
The following elementary lemma helps us to handle mixed components of those.

\begin{Lem}
\label{posidef}
Let $H = ( h_{j\ol{k}} )_{1 \leq j,k \leq n} \in M(n,\C)$ be a hermitian matrix. 
Denote $A = ( h_{j\ol{k}} )_{1 \leq j,k \leq n-1} \in M(n-1,\C)$, $b = (h_{j\ol{n}})_{1 \leq j \leq n-1} \in \C^{n-1}$, $c = h_{n\ol{n}} \in \R$. 
Then, $H$ is positive definite 
if and only if 
$cA - b \overline{ ^tb}$ is positive definite and $c > 0$. 
\end{Lem}

\begin{proof}[Proof of Theorem \ref{localformula}] 
We continue to use the local situation in the proof of Proposition \ref{converse}.

Let us consider the $(1,1)$-form 
\begin{align*}
\frac{i\d'\dbar (-\delta^\eta)}{\eta\delta^\eta} 
&= i\d'\dbar (-\log \delta) - \eta i \d'\log\delta \wedge \dbar \log \delta\\
&= \frac{i\d'\dbar (-\delta)}{\delta} + (1 - \eta) i \d'\log\delta \wedge \dbar \log \delta.
\end{align*}
and identify this with a hermitian matrix $H$ and 
denote its component by $h_{j\ol{k}}$. 
As in the proof of Proposition \ref{converse}, for $1 \leq j, k \leq n-1$, 
\begin{align*}
\lim_{y_n \searrow 0} h_{j\ol{k}}(0', iy_n) 
& = \lim_{y_n \searrow 0} \left( \frac{\d'^2 (-\log \delta)}{\d'z_j \d'\zbar_k} - \eta  \frac{\d' \log \delta}{\d'z_j} \frac{\d' \log \delta}{\d'\zbar_k}\right)\\
& = \lim_{y_n \searrow 0} \left( \frac{\d'^2 (-\log (\delta/y_n))}{\d'z_j \d'\zbar_k} - \eta  \frac{\d' \log (\delta/y_n)}{\d'z_j} \frac{\d' \log (\delta/y_n)}{\d'\zbar_k}\right)\\
& =  \frac{\d'^2 (-\log h_\delta)}{\d'z_j \d'\zbar_k}(p) - \eta  \frac{\d' \log h_\delta}{\d'z_j}(p) \frac{\d' \log h_\delta}{\d'\zbar_k}(p);
\end{align*}
for $1 \leq j \leq n-1, k = n$,
\begin{align*}
\lim_{y_n \searrow 0} y_n h_{j\ol{n}}(0', iy_n) 
& = \lim_{y_n \searrow 0} y_n \left( \frac{1}{\delta} \frac{\d'^2 (-\delta)}{\d'z_j \d'\zbar_n} + (1 - \eta) \frac{\d' \log (\delta/y_n)}{\d'z_j} \frac{1}{\delta} \frac{\d'\delta}{\d'\zbar_n}\right)\\
& = - \frac{1}{h_\delta(p)} \frac{\d'}{\d'z_j}\left( \frac{i}{2} h_\delta \right)(p) + (1 - \eta) \frac{\d' \log h_\delta}{\d'z_j}(p) \frac{1}{h_\delta(p)} \left(\frac{i}{2} h_\delta (p)\right) \\
& = \frac{\eta}{2i} \frac{\d' \log h_\delta}{\d'z_j}(p);
\end{align*}
for $j = k = n$,
\begin{align*}
\lim_{y_n \searrow 0} y_n^2 h_{n\ol{n}}(0', iy_n) 
& = \lim_{y_n \searrow 0} y_n^2 \left(\frac{1}{\delta} \frac{\d'^2 (-\delta)}{\d'z_n \d'\zbar_n} + (1 - \eta) \frac{1}{\delta^2}\left|\frac{\d' \delta}{\d'z_n}\right|^2 \right)\\
& = 0 + (1 - \eta)  \left(\frac{\d' \delta}{\d' y_n}(p)\right)^{-2} \left|\frac{\d' \delta}{\d'z_n}(p)\right|^2\\
& = \frac{1 - \eta}{4}.
\end{align*}
Note that
\[
iA_{h_\delta}(p) 
 = \sum_{j,k=1}^{n-1} \frac{\d' \log h_\delta}{\d'z_j}(p) \frac{\d' \log h_\delta}{\d'\zbar_k}(p)\, i dz_j \wedge d\zbar_k
\]
by its definition. Lemma \ref{posidef} yields a
\begin{claim}
$H(0', iy_n)$ is positive definite if and only if 
\begin{align*}
 &\frac{1}{y_n^2}\left(\frac{1-\eta}{4} + o(1)\right)\left(i\Theta_{h_\delta}(p) - \eta i A_{h_\delta}(p) + o(1)\right) - \frac{1}{y_n^2}\left(\frac{\eta^2}{4} i A_{h_\delta}(p) + o(1)\right) \\
=&\frac{1-\eta}{4y_n^2}\left( i\Theta_{h_\delta}(p) - \frac{\eta}{1-\eta} i A_{h_\delta}(p) + o(1) \right)
\end{align*}
is positive definite at $(0', iy_n)$, for $0 < y_n \ll 1$ where $o(1)$ are vanishing terms as $y_n \searrow 0$. 
\end{claim}

Now we prove the equality
\[
\eta_\delta(p) = \sup \{ \eta \in (0,1) \mid i\Theta_{h_\delta}(p) - \frac{\eta}{1-\eta} i A_{h_\delta}(p) > 0 \}
\]
by showing the inequalities for both directions. 

From the claim, $\eta < \eta_\delta(p)$ implies $i\Theta_{h_\delta}(p) - \frac{\eta}{1-\eta} i A_{h_\delta}(p) \geq 0$.
Note that $i\Theta_{h_\delta}(p) - \varepsilon i A_{h_\delta}(p) \geq 0$
implies 
$i\Theta_{h_\delta}(p) - \varepsilon' i A_{h_\delta}(p) > 0$ for $0 \leq \varepsilon' < \varepsilon$
since $i\Theta_{h_\delta}(p) > 0$ from Proposition \ref{converse}.
We thus have 
\begin{align*}
\eta_\delta(p) 
&\leq \sup \{ \eta \in (0,1) \mid i\Theta_{h_\delta}(p) - \frac{\eta}{1-\eta} iA_{h_\delta}(p) \geq 0 \}\\
&= \sup \{ \eta \in (0,1) \mid i\Theta_{h_\delta}(p) - \frac{\eta}{1-\eta} iA_{h_\delta}(p) > 0 \}.
\end{align*}

To show the other inequality, first we see from the claim that 
$i\Theta_{h_\delta}(p) - \frac{\eta}{1-\eta} iA_{h_\delta}(p) > 0$ implies 
$i\d'\dbar(-\delta^\eta)$ is strictly plurisubharmonic on $\Gamma_p$, a neighborhood of $\{(0', iy_n) \mid 0 < y_n \ll 1 \}$ in $\Omega$.
From the continuity of $i\Theta_{h_\delta}$ and $iA_{h_\delta}$, 
we can find a neighborhood $W \subset M $ of $p$ so that $i\Theta_{h_\delta}(q) - \frac{\eta}{1-\eta} iA_{h_\delta}(q) > 0$ for any $q \in W$. 
We apply the same argument at each $q \in W$ and 
conclude that $i\d'\dbar(-\delta^\eta)$ is strictly plurisubharmonic on $\bigcup_{q \in W} \Gamma_q$,
which yields 
\begin{align*}
\eta_\delta(p) 
&\geq \sup \{ \eta \in (0,1) \mid i\Theta_{h_\delta}(p) - \frac{\eta}{1-\eta} iA_{h_\delta}(p) > 0 \}.
\end{align*}
The proof is completed.
\end{proof}

\section{The exponent of conformal harmonic measures}
\label{Srelation}

\subsection{Conformal harmonic measure}
Let us briefly recall the notion of conformal harmonic measure following Brunella \cite{brunella2007}. 
For the notions of foliated harmonic measure and current, and their applications, 
we refer the reader to Fornaess and Sibony \cite{fornaess-sibony2008} and Deroin \cite{deroin-lectnote}.

We work with a complex surface $X$ and $\Omega \Subset X$ with $\Cont^3$-smooth Levi-flat boundary $M$.
We endow the Levi foliation with a continuous leafwise hermitian metric $\omega$ and 
its associated volume form $\mathrm{vol}_\omega$.
We restrict ourselves to describe conformal harmonic measures in this setting. 

\begin{Def}
Let $\mathcal{M} \subset M$ be a compact saturated set of the Levi foliation of $M$. 
Consider a finite cover $\{ U \}$ of $M$ by foliated charts 
$(z'_U, t_U): U \to D_U \times T_U \subset \C \times \R $ 
and denote the transversals of $\mathcal{M}$ by $\mathcal{M}_U \eqdef{t_U}(\mathcal{M} \cap U)$. 
A collection of transverse measures $\{\nu_U\}$ is said to be {\em conformal of exponent} $\alpha$ ($>0$) {\em adapted to} $\mathcal{M}$
if 
\begin{itemize}
\item Each $\nu_U$ is a finite Borel measure supported on $\mathcal{M}_U$;
\item $\{\nu_U\}$ relate in the following manner
\[
t_V^* (\nu_{V}) = \left|\frac{dt_V}{dt_U}\right|^\alpha \cdot \nu_U
\]
when two foliated charts $U$ and $V$ intersect.
\end{itemize}
\end{Def}

The example in mind is the $\alpha$-dimensional Hausdorff measure $\mathcal{H}_\alpha$ 
on transversals $\mathcal{M}_U$ with $0 < \mathcal{H}_\alpha(\mathcal{M}_U) < \infty$, 
then the exponent is just $\alpha$. 

\begin{Def}
A measure $\mu$ supported on a compact saturated set $\mathcal{M} \subset M$ is said to be a 
{\em conformal harmonic measure of exponent} $\alpha$ 
if there exists a finite cover $\{ U \}$ of $M$ by foliated charts 
with which $\mu$ has a local disintegration $\mu|U = h_U \mathrm{vol}_\omega \otimes \nu_U$ on each $U$ 
such that 
\begin{itemize}
\item $\{ \nu_U \}$ is a collection of transverse measures and conformal of exponent $\alpha$
      adapted to $\mathcal{M}$;  
\item $h_U$ is a non-negative real-valued function on $U$ which is integrable with respect to $\mathrm{vol}_\omega \otimes \nu_U$,
      and $\log h_U$ is bounded on $\mathcal{M} \cap U$;
\item On $\nu_U$-almost every plaque of $\mathcal{M}$, $h_U$ is positive and $\d'_b \dbar_b$-closed.  
\end{itemize}
\end{Def}



\subsection{The Diederich--Fornaess exponent and the exponent of conformal harmonic measures}
We are going to prove Proposition \ref{relation}.

\begin{proof}[Proof of Proposition \ref{relation}]
Let us suppose the existence of a $\Cont^3$-smooth boundary distance function $\delta$ 
with $\eta_\delta(p) = \eta$ on $\mathcal{M}$.
Consider the hermitian metric $h_\delta^\alpha$ of $(N^{1,0})^{\otimes \alpha}$ induced from $h_\delta$ of $N^{1,0}$, 
where we put $\alpha \eqdef (\eta^{-1} - 1)^{-1}$.
Note that $N^{1,0}$ is a flat CR line bundle over $M$ with structure group $\R_{>0}$, 
and its power $(N^{1,0})^{\otimes \alpha}$ is well-defined for $\alpha \in \R$.
From its transition rule, we have a well-defined measure $\mu$ supported on $\mathcal{M}$ by letting  
$\mu = (h^\alpha \mathrm{vol}_\omega \otimes \mathcal{H}^\alpha)|\mathcal{M}_U$ on each foliated chart $U$.

The measure $\mu$ is a conformal harmonic measure of exponent $\alpha$:
We take an arbitrary finite cover of $M$ by foliated charts and 
disintegrate $\mu$ as in its definition. 
Since $0 < \mathcal{H}_\alpha(\mathcal{M}_U) < \infty$ is supposed on each foliated chart $U$,
we can check the requirements for $\mu$ to be a conformal harmonic measure 
except $\d'_b\dbar_b$-closedness of $h_\delta^\alpha$ on plaques contained in $\mathcal{M}$.  
It follows from Theorem \ref{localformula} that
\[
\left(\frac{i\d'_b\dbar_b (-h_\delta)}{h_\delta} + 2 {i\d'_b \log h_\delta \wedge \dbar_b \log h_\delta} \right) \eta 
=
\left(\frac{i\d'_b\dbar_b (-h_\delta)}{h_\delta} + {i\d'_b \log h_\delta \wedge \dbar_b \log h_\delta} \right)
\]
everywhere on $\mathcal{M}$, and this yields that  
\begin{align*}
\d'_b\dbar_b (h_\delta^{\alpha})
&= -\alpha h_\delta^{\alpha} \left( \frac{\d'\dbar (-h_\delta)}{h_\delta} + (1 - \alpha) \frac{\d' h_\delta \wedge \dbar h_\delta}{h_\delta^2} \right) \\
&= -\alpha h_\delta^{\alpha} \left( \frac{\d'\dbar (-h_\delta)}{h_\delta} + \frac{1 - 2\eta}{1 - \eta} \frac{\d' h_\delta \wedge \dbar h_\delta}{h_\delta^2} \right)\\ 
&= 0.
\end{align*}
\end{proof}

\begin{Rem}
The obtained conformal harmonic measure is very special one since it has $\Cont^2$-smooth local coefficient $h^\alpha_\delta$.
\end{Rem}

\begin{Rem}
An expected condition $\alpha \leq 1$ and the formula $\alpha = (\eta^{-1} - 1)^{-1}$ suggest that $\eta \leq 1/2$. 
This actually holds in a general context; we refer the reader to our forthcoming paper \cite{adachi-brinkschulte2014}
and a paper of Fu and Shaw \cite{fu-shaw2014}.
\end{Rem}

We conclude this paper with an example that illustrates the situation in Proposition \ref{relation}. 

\begin{Ex}
Let $\Sigma$ be a compact Riemann surface of genus $\geq 2$. 
Fix an identification of its universal cover $\widetilde{\Sigma}$ with the unit disk $\D$
and express $\Sigma = \D/\Gamma$ by a Fuchsian group $\Gamma$.

We consider a compact complex surface $X = \D \times \CP^1 / \sim$ 
where $(z', \zeta) \sim (\gamma z', \gamma \zeta)$ for $\gamma \in \Gamma$
and a domain $\Omega = \D \times \D / \sim \; \subset X$. 
The boundary $M = \bd \Omega = \D \times \bd\D / \sim$ is real-analytic Levi-flat since the horizontal foliation 
of $\D \times \bd\D$ induces a foliation by complex curves on $M$. We equip the Poincar\'e 
metric $\omega$ on each leaf by using the quotient map. 

A boundary distance function $\delta$ of $\Omega$ is induced from a function   
\[
\delta(z', \zeta) = 1 - \left| \frac{\zeta - z'}{1 - \overline{z'} \zeta} \right|^2
\]
on $\D \times \C$, in which coordinate we will work from here. 
The induced metric $h_\delta$ of $N^{1,0}$ enjoys $i\Theta_{h_\delta}(0, \zeta) = iA_{h_\delta}(0, \zeta) = idz' \wedge d\ol{z'}$
and it follows that the local Diederich--Fornaess exponent of $\delta$ is constant equal to 1/2. 

The corresponding exponent $\alpha$ equals 1. Hence, a conformal harmonic measure $\mu$ with exponent 1
 is induced on $M$ from Proposition \ref{relation}, that is, 
\[
\mu = 2\frac{1-|z'|^2}{|e^{i\theta} - z'|^2} \mathrm{vol}_\omega \otimes d\theta
\]
where we denote $\zeta = e^{i\theta}$ on $\D \times \bd\D$. 
This expression agrees with the well-known one up to a multiplicative constant. 
\end{Ex}


\subsection*{Acknowledgments}
The author would like to thank K. Matsumoto and T. Ohsawa for their useful comments on
techniques to handle the Levi form and related literatures respectively.  
Part of this work was done during ``Geometry and Foliations 2013'' at the University of Tokyo;
The author is grateful to the organizers for their support and providing the stimulating environment.

He also thanks A.-K. Herbig and S. Biard for valuable discussions on Remark 2.7 and 2.8 respectively, 
and Ninh V. T. for pointing out mistakes before the proofreading.

\end{document}